\documentclass[a4paper,12pt]{article}
\usepackage[centertags]{amsmath}
\usepackage{amsfonts}
\usepackage{amssymb}
\usepackage{amsthm}
\usepackage{dsfont}
\usepackage{tikz, subfigure}
\usepackage{verbatim}

\newtheorem{theorem}{Theorem}[section]
\newtheorem{lemma}{Lemma}[section]

\newtheorem{defn}{Definition}[section]

\begin{document}
\title{xxxx}
\date{\today}
 \title{On the Invariant Density of the Random $\beta$-Transformation}
\author{Tom Kempton}
\maketitle
\begin{abstract}
\noindent We construct a Lebesgue measure preserving natural extension of the random $\beta$-transformation $K_{\beta}$. This allows us to give a formula for the density of the absolutely continuous invariant probability measure of $K_{\beta}$, answering a question of Dajani and de Vries, and also to evaluate some estimates on the typical branching rate of the set of $\beta$-expansions of a real number. 
\end{abstract}
\section{Introduction}
Given real numbers $\beta>1$ and $x\in I_{\beta}:=\left[0,\frac{\lfloor\beta\rfloor}{\beta-1}\right]$, a $\beta$-expansion of $x$ is a sequence $(a_i)_{i=1}^{\infty} \in \{0,\cdots,\lfloor\beta\rfloor\}^{\mathbb N}$ such that
\[
x=\sum_{i=1}^{\infty}a_i\beta^{-i}.
\]
For $\beta>1$ and $x\in I_{\beta}$ we let $\mathcal E_{\beta}(x)$ be the set of $\beta$-expansions of $x$. The study of $\beta$-expansions goes back to Renyi \cite{Renyi} and Parry \cite{Parry}, who were interested in the properties of the lexicographically largest $\beta$-expansion of $x$, known as the greedy $\beta$-expansion. It was shown that the greedy expansion $(a_i)_{i=1}^{\infty}$ of $x\in[0,1]$ can be generated by defining $T(x)=\beta x$ (mod $1$) and letting $a_i=k$ whenever $T^{i-1}(x)\in\left[\frac{k-1}{\beta},\frac{k}{\beta}\right).$ Furthermore, it was shown that $T$ preserves an absolutely contiuous probability measure which one can use in the study the ergodic properties of typical greedy $\beta$-expansions.

More recently, several authors have studied the set $\mathcal E_{\beta}(x)$ of all $\beta$-expansions of $x$. There has been substantial interest in understanding the cardinality of $\mathcal E_{\beta}(x)$ and in giving conditions under which the $\beta$-expansion of $x$ is unique. Typically $\mathcal E_{\beta}(x)$ is uncountable, see \cite{Sidorov1}, and in that case it is interesting to study the branching rate of $\mathcal E_{\beta}(x)$, which is the growth rate of the number of words $a_1\cdots a_n$ which can be continued to give $\beta$-expansions of a given $x$. In \cite{DKRandom} Dajani and Kraaikamp introduced the random $\beta$-transformation $K_{\beta}$, which allows one to generate $\mathcal E_{\beta}(x)$ dynamically, and this has allowed for a very successful dynamical approach to the study of $\mathcal E_{\beta}(x)$, see for example \cite{BakerGolden,BakerGrowth,DdV,DKRandom,FengSidorov,CountingBeta,Sidorovlmt}. 

The ergodic theory of $K_{\beta}$ was investigated in \cite{DdV1} and \cite{DdV}, where two natural invariant measures were found. Links between the measure of maximal entropy $\hat{\nu}_{\beta}$ described in \cite{DdV1}, counting $\beta$-expansions and the question of absolute continuity Bernoulli convolutions provide some motivation for this work and are explained in the next section. However our main focus is on the absolutely continuous invariant measure $\hat{\mu}_{\beta}$ of $K_{\beta}$ which was described by Dajani and de Vries in \cite{DdV}. They gave a formula for the density of $\hat{\mu}_{\beta}$ in some special cases. In this article we build a natural extension of the system $(K_{\beta},\hat{\mu}_{\beta})$, which allows us to recover a formula for the density of $\hat{\mu}_{\beta}$ in the general case, providing a solution to one of the open problems stated in \cite{DdV}. 


In section $2$ we define the random $\beta$-transformation and give the formula for the density of $\hat{\mu}_{\beta}$. In section \ref{greedysection} we recall the natural extension of the greedy $\beta$-transformation which serves as our starting point. We generalise this natural extension of the greedy $\beta$-transformation in section 4 to build a tower and a dynamical system, but for some technical reasons this tower does not serve as a natural extension of $(K_{\beta},\hat{\mu}_{\beta})$. Finally in section 5 we adapt our construction from section 4 to build our natural extension.

\subsection{Bernoulli Convolutions and Counting $\beta$-expansions}
In addition to gaining a better understanding of the random $\beta$-transformation, our work allows us to draw conclusions for typical $x$ about the set $\mathcal E_{\beta}(x)$ of $\beta$-expansions of $x$. In \cite{CountingBeta} we gave a lower bound for the typical branching rate (or equivalently the Hausdorff dimension) of the set $\mathcal E_{\beta}(x)$ in terms of $\hat{\mu}_{\beta}$, using the formula for the density of $\hat{\mu}_{\beta}$ obtained in this article we can make this lower bound explicit. This in turn is relevant to the study of Bernoulli convolutions.


Bernoulli convolutions are self similar measures with overlaps. Given $\beta\in(1,2)$ we define $\pi_{\beta}:\{0,1\}^{\mathbb N}\to I_{\beta}$ by
\[
\pi_{\beta}(\underline a)=\sum_{i=1}^{\infty}a_i\beta^{-i}.
\]
The Bernoulli convolution is the probability measure on $I_{\beta}$ defined by
\[
\nu_{\beta}=m\circ\pi_{\beta}^{-1}
\]
where $m$ is the $(\frac{1}{2},\frac{1}{2})$ Bernoulli measure on $\{0,1\}^{\mathbb N}$. It is a difficult open question to determine the parameters $\beta$ for which $\nu_{\beta}$ is absolutely continuous, for a review see \cite{SolomyakSixty}. The measure of maximal entropy $\hat{\nu}_{\beta}$ of $K_{\beta}$ projects to the Bernoulli convolution $\nu_{\beta}$ on its second coordinate\footnote{In this work measures $\hat{\mu}_{\beta}, \hat{\nu}_{\beta}$ are two dimensional and supported on the domain of $K_{\beta}$, whereas $\mu_{\beta}$ and $\nu_{\beta}$ denote the projection of $\hat{\mu}_{\beta}, \hat{\nu}_{\beta}$ onto the second coordinate.}.



In \cite{JordanShmerkinSolomyak} the sets $\mathcal E_{\beta}(x)$ were used in the multifractal analysis of Bernoulli convolutions. Furthermore, in \cite{CountingBeta} we gave sufficient conditions for the absolute continuity of Bernoulli convolutions in terms of some counting questions relating to $\mathcal E_{\beta}(x)$. It is perhaps unsurprising that the nature of the Bernoulli convolution is given by the typical properties of the sets $\mathcal E_{\beta}(x)$, since $\nu_{\beta}$ is a projection of the measure $m$ by $\pi_{\beta}$, and the sets $\mathcal E_{\beta}(x)$ are just the preimages $\pi_{\beta}^{-1}(x)$ of points $x\in I_{\beta}$. What is more intriguing however is the idea that one can study the branching rate of $\mathcal E_{\beta}(x)$, and hence the question of the absolute continuity of $\nu_{\beta}$, without studying the difficult measures $\nu_{\beta}$ or $\hat{\nu}_{\beta}$ directly but instead through the ergodic theory of the system $(K_{\beta},\hat{\mu}_{\beta})$.

This article constitutes a first step in this direction, by giving a formula for the density of $\hat{\mu}_{\beta}$ one can make explicit a lower bound given in \cite{CountingBeta} on the branching rate of $\mathcal E_{\beta}(x)$. Since this lower bound is not sharp, we are unable to answer the question of whether any given Bernoulli convolution is absolutely continuous. However one may hope that a more subtle analysis of the branching rate of $\mathcal E_{\beta}(x)$ in terms of the ergodic theory of $(K_{\beta},\hat{\mu}_{\beta})$, coupled with the description of $\hat{\mu}_{\beta}$ given in this article, may give progress in this direction. This is discussed in the final section.



\section{The Random $\beta$-transformation}
Since we are motivated by the study of Bernoulli convolutions $\nu_{\beta}$ associated to $\beta\in(1,2)$, we restrict our study of the natural extension of $K_{\beta}$ to the case $\beta\in(1,2)$. The extension to general $\beta>1$ is straightforward, although the notation involved is more complicated.

We partition the interval $\left[0,\frac{1}{\beta-1}\right]$ into the sets \[L=\left[0,\frac{1}{\beta}\right), S=\left[\frac{1}{\beta},\frac{1}{\beta(\beta-1)}\right]\text{ and }R=\left(\frac{1}{\beta(\beta-1)},\frac{1}{\beta-1}\right].\] We let $T_0,T_1:\mathbb R\rightarrow\mathbb R$ be given by $T_0(x)=\beta x$ and $ T_1(x)=\beta x -1$ and let $\Omega=\{0,1\}^{\mathbb N}$. The random $\beta$-transformation  $K_{\beta}:\Omega\times[0,\frac{1}{\beta-1}]\to\Omega\times[0,\frac{1}{\beta-1}]$ is defined by 

\[
K_{\beta}(\omega,x)=\left\lbrace\begin{array}{c c}(\omega,T_0(x))& x \in L\\ (\sigma(\omega),T_{\omega_1}(x))& x \in S\\ (\omega,T_1(x)) & x \in R\end{array}\right.
\]
where $\omega=(\omega_i)_{i=1}^{\infty}$. Given a pair $(\omega,x)$, we generate a sequence $(x_n)_{n=1}^{\infty}$ by iterating $K_{\beta}(\omega,x)$. If the $n$th iteration of $K_{\beta}(\omega,x)$ applies $T_0$ to the first coordinate we put $x_n=0$, if it applies $T_1$ to the first coordinate we put $x_n=1$. The sequence $(x_n)_{n=1}^{\infty}$ is a $\beta$-expansion of $x$. Each $\beta$-expansion of $x$ can be generated by this algorithm with some choice of $\omega$, and for typical $x$ each different choice of $\omega$ corresponds to a different $\beta$-expansion of $x$.

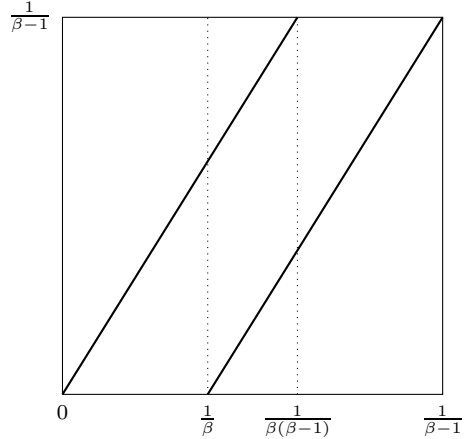
\begin{figure}[h]
\centering
\begin{tikzpicture}[scale=5]
\draw(0,0)node[below]{\scriptsize 0}--(.382,0)node[below]{\scriptsize$\frac{1}{\beta}$}--(.618,0)node[below]{\scriptsize$\frac{1}{\beta(\beta-1)}$}--(1,0)node[below]{\scriptsize$\frac{1}{\beta-1}$}--(1,1)--(0,1)node[left]{\scriptsize$\frac{1}{\beta-1}$}--(0,.5)--(0,0);
\draw[dotted](.382,0)--(.382,1)(0.618,0)--(0.618,1);
\draw[thick](0,0)--(0.618,1)(.382,0)--(1,1);
\end{tikzpicture}\caption{The projection onto the second coordinate of $K_{\beta}$ for $\beta=\dfrac{1+\sqrt{5}}{2}$}
\end{figure}

In \cite{DdV}, Dajani and de Vries showed that $K_{\beta}$ has an invariant probability measure $\hat{\mu}_{\beta}=m_{\frac{1}{2}}\times\mu_{\beta}$, where $\mu_{\beta}$ is absolutely continuous with respect to Lebesgue measure and $m_{\frac{1}{2}}$ is the $\left(\frac{1}{2},\frac{1}{2}\right)$ Bernoulli measure on $\Omega$. They also showed that $K_{\beta}$ is ergodic with respect to this measure.\footnote{In fact Dajani and de Vries also proved the existence of invariant probability measures $\hat{\mu}_{\beta,p}=m_p\times\mu_{\beta,p}$ where $\mu_{\beta,p}$ is absolutely continuous and $m_p$ is the $(p,1-p)$ Bernoulli measure on $\Omega$. In this article we deal only with the unbiased case $p=\frac{1}{2}$.}

Many properties of $K_{\beta}$ can be studied using the related skew product transformation $R_{\beta}$. We define $R_{\beta}:\Omega\times[0,\frac{1}{\beta-1}]\to\Omega\times[0,\frac{1}{\beta-1}]$ by
\[
R_{\beta}(\underline{\omega},x)=\left\lbrace \begin{array}{cc}(\sigma(\underline{\omega}),T_0 (x))& x\in L\\ (\sigma(\underline{\omega}),T_{\omega_1}(x))& x\in S\\ (\sigma(\underline{\omega}),T_1(x)) & x\in R\end{array} \right. .
\]
In particular, the measure $\hat{\mu}_{\beta}$ is invariant under $R_{\beta}$. We build a natural extension for the system $(\Omega\times I_{\beta},R_{\beta}, \hat{\mu}_{\beta})$, this can be easily translated to a natural extension for $K_{\beta}$ by changing when one shifts in the first coordinate, but we present the case of $R_{\beta}$ as the notation is easier and it gives us the same information about $\hat{\mu}_{\beta}$.

We will often be interested in the second coordinate of $R_{\beta}^n(\omega,x)$. We introduce the shorthand $\pi_2(\omega,x):=x$, $R_{\beta,\omega}(x):=\pi_2(R_{\beta}(\omega,x))$. Since $R_{\beta,\omega}^n(x)$ depends only on the first $n$ coordinates of $\omega$, we sometimes write $R_{\beta,\omega_1\cdots\omega_n}^n(x)$

We recall that Parry \cite{Parry} proved that the absolutely continuous invariant measure of the map $T(x)=\beta x$ (mod 1) has density proportional to
\[
d(x)=\sum_{n=0}^{\infty} \frac{1}{\beta^{n}}\chi_{[0,T^n(1)]}(x).
\]

The importance of the orbit of $1$ in determining the invariant measure for $T$ is due to the fact that $1$ is the limit of $T(x)$ as $x$ approaches $\frac{1}{\beta}$ from below, and $\frac{1}{\beta}$ is the point of discontinuity for the system. $R_{\beta}$ is discontinuous when $x=\frac{1}{\beta}$ and $x=\frac{1}{\beta(\beta-1)}$, and so one may expect that the orbits of $1$ and $\frac{1}{\beta-1}-1$ may play a similar role in determining the invariant density for $R_{\beta}$. Furthermore, since the points $1$ and $\frac{1}{\beta-1}-1$ have (typically) uncountably many orbits associated to different choices of $\omega\in\Omega$, we should expect each of these different orbits to have some role in determining $\hat{\mu}_{\beta}$. 

The following theorem confirms this intuition; the invariant density for $R_{\beta}$ can be obtained by modifying the formula of Parry to take in to account orbits of the point $\frac{1}{\beta-1}-1$ and allowing for different orbits corresponding to different choices of $\omega$.
\begin{theorem}\label{densitytheorem}
The density of $\mu_{\beta}$ is given by
\[
\rho_{\beta}(x)=C\sum_{n=0}^{\infty}\frac{1}{(2\beta)^n}\left(\sum_{\omega_1\cdots\omega_n\in\{0,1\}^n} \chi_{[0,R_{\beta,\omega_1\cdots\omega_n}^n(1)]}(x)+\chi_{[R_{\beta,\omega_1\cdots\omega_n}^n(\frac{1}{\beta-1}-1),\frac{1}{\beta-1}]}(x)\right).
\]
where $C$ is just a normalising constant to make $\mu_{\beta}$ a probability measure.
\end{theorem}
The exponential decay in the summand allows us to estimate the density with explicit error bounds. The proof of Theorem \ref{densitytheorem} is done via the construction of a natural extension of $R_{\beta}$, which occupies the majority of this article.


\section{The natural extension of the greedy map}\label{greedysection}
Our method is reminiscent of the natural extension of the greedy $\beta$-transformation given by Dajani, Kraaikamp and Solomyak \cite{DKS}, see also \cite{DajaniKalleDeleted} for a related construction on greedy $\beta$-transformations with deleted digits. We begin by recalling the approach of \cite{DKS}. The authors built a tower as a natural extension of the map $T(x)=\beta x $ (mod 1) on $[0,1]$ and let the $n$th level of the tower be given by
\[
X_n:=[0,T^n(1)]\times\left[\sum_{i=0}^{n-1}\beta^{-i},\sum_{i=0}^{n}\beta^{-i}\right)
\]
with $X_0=[0,1]\times[0,1]$. The levels of the tower stack neatly on top of each other. The domain $X$ of the natural extension is given by $X=\cup_{n=0}^{\infty} X_n$, and the transformation is defined in terms of the orbit of $1$.

When $T^{n+1}(1)=\beta T^n(1)$, $X_n$ is mapped bijectively onto $X_{n+1}$ by $(x,y)\to (\beta x, \frac{y}{\beta}+1)$.

When $T^{n+1}(1)=\beta T^n(1)-1$, $X_n$ is split into two, the set $\{(x,y)\in X_n:x\geq \frac{1}{\beta}\}$ is mapped bijectively onto $X_{n+1}$ by $(x,y)\to\left(\beta x-1,\frac{y}{\beta}+1\right)$.

The set $\{(x,y)\in X_n:x\in[0,\frac{1}{\beta})\}$ is mapped to a horizontal strip $X_0^n$ across $X_0$ of width $1$ and height $\frac{1}{\beta^{n+1}}$. This happens by applying $T_0$ to the first coordinate, dividing by $\beta$ in the second coordinate, and translating the second coordinate so that $X_0^n$ lies exactly on top of the image of $X_0^m$, where $m$ is the greatest integer less than $n$ for which $T^{m+1}(1)=\beta T^m(1)-1$. The first few levels of the tower for $\beta=1.25$ are given in Figure 2.

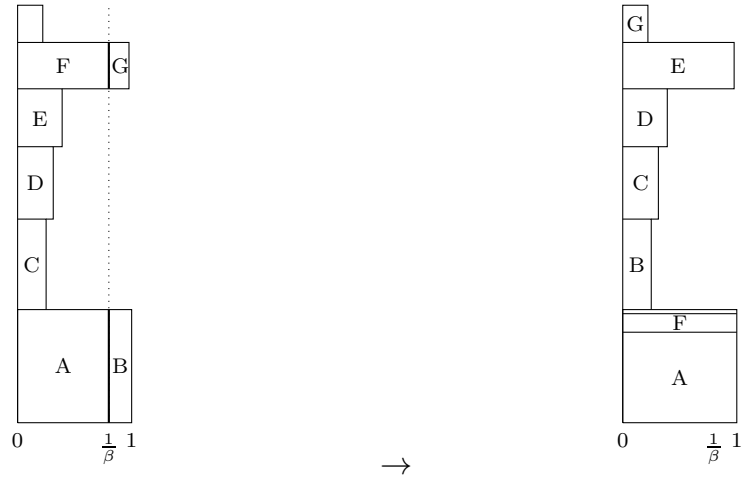
\begin{figure}[ht]\label{greedytower}
\begin{minipage}[b]{0.5\linewidth}
\centering
\begin{tikzpicture}[scale=1.5]
\draw(0,0)node[below]{\scriptsize 0}--(0.8,0)node[below]{\scriptsize$\frac{1}{\beta}$}--(1,0)node[below]{\scriptsize$1$}--(1,1)--(0,1)--(0,0);
\draw(.25,1)--(.25,1.8)--(0,1.8);
\draw(.25,1.8)--(.3125,1.8)--(0.3125,2.44)--(0,2.44);
\draw(.3125,2.44)--(.3906,2.44)--(0.3906,2.952)--(0,2.952);
\draw(0.3906,2.952)--(0.9766,2.952)--(.9766,3.3616)--(0,3.3616);
\draw(.2207,3.3616)--(.2207,3.6906)--(0,3.6906)--(0,0);
\draw[dotted](0.8,0)--(0.8,3.6906);
\draw[thick](0.8,0)--(0.8,1) (0.8,2.952)--(0.8,3.3616);
\draw(0.4,0.5)node{\scriptsize A} (0.9,0.5)node{\scriptsize B} (0.125,1.4)node{\scriptsize C} (0.15625,2.12)node{\scriptsize D} (0.1953,2.69)node{\scriptsize E} (0.4,3.1568)node{\scriptsize F} (0.9,3.1568)node{\scriptsize G};
\end{tikzpicture}
\end{minipage}
\hspace{0.5cm}$\to$
\begin{minipage}[b]{0.5\linewidth}
\centering
\begin{tikzpicture}[scale=1.5]
\draw(0,0)node[below]{\scriptsize 0}--(0.8,0)node[below]{\scriptsize$\frac{1}{\beta}$}--(1,0)node[below]{\scriptsize$1$}--(1,1)--(0,1)--(0,0);
\draw(.25,1)--(.25,1.8)--(0,1.8);
\draw(.25,1.8)--(.3125,1.8)--(0.3125,2.44)--(0,2.44);
\draw(.3125,2.44)--(.3906,2.44)--(0.3906,2.952)--(0,2.952);
\draw(0.3906,2.952)--(0.9766,2.952)--(.9766,3.3616)--(0,3.3616);
\draw(.2207,3.3616)--(.2207,3.6906)--(0,3.6906)--(0,0);
\draw(0,0.8)--(1,0.8) (0,0.9638)--(1,0.9638);
\draw(0.5,0.4)node{\scriptsize A} (0.5,0.8819)node{\scriptsize F} (0.125,1.4)node{\scriptsize B} (0.15625,2.12)node{\scriptsize C} (0.1953,2.69)node{\scriptsize D} (0.4833,3.1568)node{\scriptsize E} (0.1103,3.52608)node{\scriptsize G};
\end{tikzpicture}
\end{minipage}\caption{The first few levels of the natural extension of the $\beta$ transformation for $\beta=1.25$, rectangles in the diagram on the left are mapped to the rectangle with corresponding label in the diagram on the right.}
\end{figure}


The above system is a natural extension of the system $([0,1],T,\mu_P)$ where $\mu_P$ is the absolutely continuous invariant measure. Since the natural extension preserves Lebesgue measure, we can recover the formula of Parry for the density of $\mu_P$ by projecting Lebesgue measure on the tower $X$ down to the unit interval.

\section{A Tower for the Random $\beta$-transformation}
Following \cite{DKS}, we build a tower and a dynamical system related to $R_{\beta}$ using the orbit of $1$. We begin by generalising the method of \cite{DKS} directly to the case that the orbit of $1$ may depend on $\omega$. In fact, a modification will be needed in order to make this dynamical system a natural extension of $R_{\beta}$, this is deferred until the final section.

For a typical $\beta\in(1,2)$ there is not a single orbit of $1$ under the random $\beta$-transformation but uncountably many orbits associated to $(\omega,1)$ for different $\omega$. Consequently we have to split the $n$th level of the tower into $2^n$ sublevels $E_{\omega_1\cdots\omega_n}$ associated to each choice of $\omega_1\cdots\omega_n$. For $n\in\mathbb N$ we order the sublevels of the $n$th level of the tower by letting 
\[
l(\omega_1\cdots\omega_n)=\sum_{i=1}^n\omega_i2^{i-1}\in\{1,\cdots 2^n\}.
\]

We let the height of the sublevel of $E_n$ assocated to $\omega_1\cdots\omega_n$ be given by \[v(\omega_1\cdots\omega_n)=\sum_{k=0}^{n-1}\beta^{-k}+\frac{l(\omega_1\cdots\omega_n)-1}{(2\beta)^n}.\]

Then the set of intervals 
\[
\left\{\left[v(\omega_1\cdots\omega_n),v(\omega_1\cdots\omega_n)+\frac{1}{(2\beta)^n}\right):\omega_1\cdots\omega_n\in\{0,1\}^n\right\}
\]
partition the interval $[\sum_{i=0}^{n-1}\frac{1}{\beta^{i}},\sum_{i=0}^{n}\frac{1}{\beta^{i}}]$, this interval will correspond to the $y$-coordinates of the $n$th level of the tower.

The right end points of the sublevels of our tower are given in terms of the orbit of $1$ under $R_{\beta,\omega}$, we define
\[
r(\omega_1\cdots\omega_n):=R_{\beta,\omega_1\cdots\omega_n}^n(1).
\]

Then for $\omega_1\cdots\omega_n\in\{0,1\}^n$ we define the set 
\[
E_{\omega_1\cdots\omega_n}:=\Omega\times[0,r(\omega_1\cdots\omega_n)]\times\left[v(\omega_1\cdots\omega_n),v(\omega_1\cdots\omega_n)+\frac{1}{(2\beta)^n}\right).
\]
Finally we define $E_{base}=\Omega\times[0,1]\times[0,1)$ and the tower \[E:=E_{base}\cup\left(\bigcup_{n=1}^{\infty}\bigcup_{\omega_1\cdots\omega_n\in\{0,1\}^n} E_{\omega_1\cdots\omega_n}\right).\]

This resembles the tower for the greedy $\beta$-transformation except that the $n$th level is split into different sublevels corresponding to the different orbits of $1$, and there is an extra first coordinate corresponding to the sequences $\omega\in\Omega$.

\subsection{Dynamics on the Tower}
We define a map $\psi$ on the tower $E$. In principle, $\psi$ works exactly the same way as the natural extension of the greedy map given in section \ref{greedysection}, we define $\psi(\sigma^n(\omega),x,y)$ based on the action of $R_{\beta,\omega_{n+1}}$ on the right end point of the sublevel of the tower to which $(\sigma^n(\omega),x,y)$ belongs. 
\begin{enumerate}
\item If $R_{\beta,\omega_{n+1}}$ acts by $T_0$ then $[\omega_{n+1}]\cap E_{\omega_1\cdots\omega_n}$ is mapped bijectively onto $E_{\omega_1\cdots\omega_{n+1}}$
\item If $R_{\beta,\omega_{n+1}}$ acts by $T_1$ then $[\omega_{n+1}]\cap E_{\omega_1\cdots\omega_n}$ is split into two pieces, one of which is mapped bijectively onto $E_{\omega_1\cdots\omega_{n+1}}$ and one of which is mapped back to $E_{base}$.
\end{enumerate}

{\bf Case 1:} If $R_{\beta,\omega_{n+1}}(r(\omega_1\cdots\omega_n))=T_0(r(\omega_1\cdots\omega_n))$ then map $[\omega_{n+1}]\cap E_{\omega_1\cdots\omega_n}$ onto $E_{\omega_1\cdots\omega_{n+1}}$ by shifting the first coordinate, applying $T_0$ to the second and shrinking by $\frac{1}{2\beta}$ and translating in the third coordinate.

More precisely, we define $C_1$ by
\[
C_1(\omega_1\cdots\omega_{n+1})=v(\omega_1\cdots\omega_{n+1})-\frac{v(\omega_1\cdots\omega_n)}{2\beta}
\]
and then define $\psi:[\omega_{n+1}]\cap E_{\omega_1\cdots\omega_n}\to E_{\omega_1\cdots\omega_{n+1}}$ by
\[
\psi(\sigma^n(\omega),x,y)=\left(\sigma^{n+1}(\omega),T_0(x),\frac{y}{2\beta}+C_1(\omega_1\cdots\omega_{n+1})\right).
\]

We stress that one can recover $\omega_1\cdots \omega_n$ by knowing which level $E_{\omega_1\cdots\omega_n}$ the triple $(\sigma^n(\omega),x,y)$ lies, thus $C_1$ and $\psi$ are well defined.

We see that $T_0[0,r(\omega_1\cdots\omega_n)]=[0,\beta r(\omega_1\cdots\omega_n)]=[0,r(\omega_1\cdots\omega_{n+1})]$, and that
\begin{eqnarray*}
& &\frac{1}{2\beta}\left[v(\omega_1\cdots\omega_n),v(\omega_1\cdots\omega_{n})+\frac{1}{(2\beta)^n}\right)+C_1(\omega_1\cdots\omega_{n+1})\\
&=&\left[v(\omega_1\cdots\omega_{n+1}),v(\omega_1\cdots\omega_{n+1})+\frac{1}{(2\beta)^{n+1}}\right),
\end{eqnarray*}
making the map  $\psi:[\omega_{n+1}]\cap E_{\omega_1\cdots\omega_n}\to E_{\omega_1\cdots\omega_{n+1}}$ a bijection.

{\bf Case 2:} If $R_{\beta,\omega_{n+1}}(r(\omega_1\cdots\omega_n))=T_1(r(\omega_1\cdots\omega_n))$ then split $[\omega_{n+1}]\cap E_{\omega_1\cdots\omega_n}$ into two pieces.

We let the part with $x$ coordinates in $S\cup R$ be mapped bijectively onto $E_{\omega_1\cdots\omega_{n+1}}$ by
\[
\psi(\sigma^n(\omega),x,y)=\left(\sigma^{n+1}(\omega),T_1(x),\frac{y}{2\beta}+C_1(\omega_1\cdots\omega_{n+1})\right),
\]
as in case 1.

We map the part with $x$ coordinates in $L$ back down into a horizontal strip of height $\frac{1}{(2\beta)^{n+1}}$ across $E_{base}$. We define the constant
\[
C_2(\omega_1\cdots\omega_{n+1})=\frac{1}{2\beta}+\left(\sum_{\substack{a_1\cdots a_{m+1}:v(a_1\cdots a_m)<v(\omega_1\cdots\omega_{n})\\ r(a_1\cdots a_{m+1})=T_1(r(a_1\cdots a_m))}}\frac{1}{(2\beta)^{m+1}}\right)-\frac{v(\omega_1\cdots\omega_n)}{2\beta},
\]
which is chosen so that the image of $[\omega_{n+1}]\cap E_{\omega_1\cdots\omega_n}\cap\{x\in L\}$ under $\psi$ lies exactly on top of all the previous pieces which have been mapped back into $E_{base}$ in the $y$ direction.


We define $\psi:[\omega_{n+1}]\cap\{(\sigma^n(\omega),x,y)\in E_{\omega_1\cdots\omega_n}:x\in L\}\to E_{base}$ by
\[
\psi(\sigma^n(\omega),x,y):=(\sigma^{n+1}(\omega),T_0(x),\frac{y}{2\beta}+C_2(\omega_1\cdots\omega_{n+1})).
\]

We have now defined $\psi$ on all of $E$. As shorthand we partition the sets $E_{\omega_1\cdots\omega_n}$ into the set $E_{\omega_1\cdots\omega_n}^U$ of those points which are mapped up the tower (i.e. whose $y$-coordinates increase under the action of $\psi$) and the set $E_{\omega_1\cdots\omega_n}^D$ of points which are mapped down into $E_{base}$ by $\psi$. We partition $E$ into $E^U$ and $E^D$ in the same way.
\begin{lemma}\label{mp}
The transformation $\psi:E\to E$ is bijective almost everywhere and preserves measure $\tilde{\lambda}:=(m\times\lambda\times\lambda)|_E$.
\end{lemma}

\begin{proof}
The transformation shifts in the first coordinate (which expands distance by a factor of two), stretches by a factor of $\beta$ in the second coordinate, and shrinks by a factor of $\frac{1}{2\beta}$ in the third coordinate. Thus, if we can prove that $\psi$ is a bijection almost everywhere this will automatically give that it preserves the measure $\tilde{\lambda}$. We have already argued that the restriction of $\psi$ to $E^U$ is a bijection onto $E/E_{base}$. It remains to prove only that $\psi$ restricted to $E^D$ maps bijectively onto $E_{base}$. 

The constant $C_2$ ensures that $E_{\omega_1\cdots\omega_n}^D$ is mapped exactly on top of all of the rectangles which have already been mapped into $E_{base}$. So $\psi$ maps $\bigcup_{n=1}^{\infty}\bigcup_{\omega_1\cdots\omega_n\in\{0,1\}^n} E_{\omega_1\cdots\omega_n}^D$ bijectively into 
\[
\Omega\times[0,1]\times\left[0,\frac{1}{2\beta}+\sum_{n=1}^{\infty}\sum_{\omega_1\cdots\omega_n\in\{0,1\}^n:E^D_{\omega_1\cdots \omega_n}\neq\phi}\frac{1}{(2\beta)^{n+1}}\right],
\]
where the term $\frac{1}{2\beta}$ corresponds to the part of $E_{base}$ which is mapped directly back into $E_{base}$. It remains to show that \[\frac{1}{2\beta}+\sum_{n=1}^{\infty}\sum_{\omega_1\cdots\omega_n\in\{0,1\}^n:E^D_{\omega_1\cdots \omega_n}\neq\phi}\frac{1}{(2\beta)^{n+1}}=1.\]

To prove this, we first observe that our tower has finite measure since it is contained in the box $\Omega\times[0,\frac{1}{\beta-1}]\times[0,\sum_{n=0}^{\infty}\beta^{-n}]$. Each time we apply $\psi$ to a level of the tower, part of the level is mapped up to the next level while part is mapped back into $E_{base}$. Each of these maps up the tower are measure preserving bijections onto their image.We denote the $k$th level of the tower $E_k:=\cup_{\omega_1\cdots\omega_k\in\{0,1\}^k}E_{\omega_1\cdots\omega_k}$. Then the total mass of $E_k$ is equal to one minus the mass of those parts of the first $k-1$ levels of the tower which are mapped back into $E_{base}$. Mass $\frac{1}{2\beta}$ is mapped from $E_{base}$ directly back into $E_{base}$. So
\[
\tilde{\mu}(E_k)=1-\frac{1}{2\beta}-\sum_{n=1}^{k-1}\sum_{\omega_1\cdots\omega_n:E_{\omega_1\cdots\omega_n}^D\neq\phi}\frac{1}{(2\beta)^{n+1}}
\]
Then, since $\sum_{k=1}^{\infty}\tilde{\mu}(E_k)<\infty$, we see that $\tilde{\mu}(E_k)\to 0$ as $k\to\infty$, giving that
\[
\frac{1}{2\beta}+\lim_{k\to\infty}\sum_{n=1}^{k}\sum_{\omega_1\cdots\omega_n\in\{0,1\}^n:E^D_{\omega_1\cdots \omega_n}\neq\phi}\frac{1}{(2\beta)^{n+1}}=1
\]
as required.
\end{proof}

\section{A Natural Extension}
In order to build a natural extension of the map $R_{\beta}$ we need to build a dynamical system that acts the same way as $R_{\beta}$ on its first two coordinates. The system $(E,\psi)$ that we have built is heavily based on $R_{\beta}$, but we have defined $\psi$ on $(\sigma^n(\omega),x,y)\in E_{\omega_1\cdots\omega_n}$ in terms of the action of $R_{\beta,\omega_{n+1}}$ on $r(\omega_1\cdots\omega_n)$ rather than on $x$. In most situations this is sufficient and the projection onto the first two coordinates of $\psi(\sigma^n(\omega),x,y)$ is equal to $R_{\beta}(\sigma^n(\omega),x)$, but in some cases there is a discrepancy as described in the following lemma.


\begin{lemma}\label{problem}
For $(\sigma^n(\omega),x,y)\in E_{\omega_1\cdots\omega_n}$ we have that 
\[
\pi_2(\psi(\sigma^n(\omega),x,y))=\left\lbrace\begin{array}{cc}R_{\beta,\omega_{n+1}}(x)-1&x\in S, \omega_{n+1}=0\text{ and }r(\omega_1\cdots\omega_n)\in R\\ R_{\beta,\omega_{n+1}}(x)& \text{ otherwise }\end{array}\right. .
\]
\end{lemma}
\begin{proof}
We see that if $x\in L$ then the action of $\psi$ on the second coordinate is to send $x$ to $\beta x$, as required. However if $x\in S\cup R$ then $\psi$ acts on the second coordinate in the same way that $R_{\beta,\omega_{n+1}}$ acts on $r(\omega_1\cdots\omega_n)$. 

If $x\in R$ then $r(\omega_1\cdots\omega_n)$ is necessarily in $R$, and so $x$ is acted on by $x\to\beta x-1$ as required. If $x\in S$ and $r(\omega_1\cdots\omega_n)\in S$ then $x\to \beta x-\omega_{n+1}$, again as required. However, if $r(\omega_1\cdots\omega_n)\in R$ then it is always mapped to $\beta r(\omega_1\cdots\omega_n)-1$ irrespective of $\omega$, and so in the case that $\omega_{n+1}=0$ there is a discrepancy between the action of $\psi$ on the tower and the action of $R_{\beta}$.\end{proof}

We let $F_{\omega_1\cdots\omega_n}$ be the set of elements of $E_{\omega_1\cdots\omega_n}$ for which $\psi$ does not behave as a natural extension of $R_{\beta}$, i.e.
\[
F_{\omega_1\cdots\omega_n}:=\left\lbrace\begin{array}{cc}\{(\sigma^n(\omega),x,y)\in E_{\omega_1\cdots\omega_n}:\omega_{n+1}=0,x\in S\}&r(\omega_1\cdots\omega_n)\in R\\
 \phi & \text{ otherwise}.\end{array}\right. .
\]

In fact we see that $F_{\omega_1\cdots\omega_n}$ is mapped by $\psi$ to points with $x$ coordinates in $(\beta-1)S=[0,\frac{1}{\beta-1}-1]$, whereas $S$ is mapped by $R_{\beta,\omega_{n+1}}$ to $\beta S=[1,\frac{1}{\beta-1}]$. We also note that the sets $[0,\frac{1}{\beta-1}-1]$ and $[1,\frac{1}{\beta-1}]$ are reflections of each other in the central line $x=\frac{1}{2(\beta-1)}$. 

The tower that we have constructed so far consists of rectangles which are attached to the left hand side of the interval $[0,\frac{1}{\beta-1}]$ which are defined in terms of the orbits of the point $1$. Since the map $R_{\beta}$ is symmetric we could just as well have constructed a tower out of rectangles attached to the right hand side of $[0,\frac{1}{\beta-1}]$, defined in terms of the orbits of $\frac{1}{\beta-1}-1$. If we were to define a dynamical system on this new tower by reflecting $\psi$ we would have the opposite problem to that outlined in Lemma \ref{problem}, our map would sometimes map to rectangles with $x$ coordinates in $[1,\frac{1}{\beta-1}]$ whereas $R_{\beta,\omega_{n+1}}$ would map them to $[0,\frac{1}{\beta-1}-1]$. 

Our solution is to have both towers. Given $\omega=(\omega_i)_{i=1}^{\infty}\in\Omega$ we define the complementary sequence $\overline{\omega}$ by $\overline{\omega_i}=1-\omega_i$. Then for $(\omega,x,y)\in E$ we define
\[
P(\omega,x,y)=(\overline{\omega},\frac{1}{\beta-1}-x,-y).
\]
Then $P(E)$ gives a second tower $\overline E$, which is disjoint from $E$. We let $\overline E_{\overline{\omega_1}\cdots\overline{\omega_n}}=P(E_{\omega_1\cdots\omega_n})$ and $\overline F_{\overline{\omega_1}\cdots\overline{\omega_n}}=P(F_{\omega_1\cdots\omega_n})$. We extend the map $\psi$ to $\overline E$ by defining
\[
\psi(\omega,x,y)=P\circ \psi\circ P^{-1}(\omega,x,y)
\]
for $(\omega,x,y)\in\overline E$. 

We define $Q:\Omega\times\mathbb R^2\to\Omega\times\mathbb R^2$ by
\[
Q(\sigma^n(\omega),x,y)=\left\lbrace\begin{array}{cc}
                                  (\sigma^n(\omega),x+1,-y) & (\sigma^n(\omega),x,y)\in \psi(F_{\omega_1\cdots\omega_n})\\
(\sigma^n(\omega),x-1,-y) & (\sigma^n(\omega),x,y)\in \psi(F_{\overline{\omega_1}\cdots\overline{\omega_n}})\\
(\sigma^n(\omega),x,y) & \text{ otherwise }
                                 \end{array}\right.
\]

$Q$ plays the role of swapping those points defined in Lemma \ref{problem} for which $\psi$ does not behave as a natural extension of $R_{\beta}$ with the corresponding points in $\overline E$ which had an equal and opposite problem. This allows us to define our natural extension.

\begin{theorem}
The function $\tilde \psi:E\cup\overline E\to E\cup\overline E$ defined by 
\[\tilde \psi= Q\circ \psi\] is a natural extension of $(\Omega\times I_{\beta}, R_{\beta}, \hat{\mu}_{\beta})$.
\end{theorem}

This is proved by the following two lemmas.
\begin{lemma}
For $(\sigma^n(\omega),x,y)\in E\cup\overline E$ we have $\pi_2(\tilde \psi(\sigma^n(\omega),x,y))=R_{\beta,\omega_{n+1}}(x)$.
\end{lemma}
\begin{proof}
Suppose that $(\sigma^n(\omega),x,y)\in E_{\omega_1\cdots\omega_n}\setminus F_{\omega_1\cdots\omega_n}$. Then by the definition of $Q$ and by Lemma \ref{problem} we have that
\[
\pi_2(\tilde \psi(\sigma^n(\omega),x,y))=\pi_2(\psi(\sigma^n(\omega),x,y))=R_{\beta,\omega_{n+1}}(x).
\]
Conversely, if $(\sigma^n(\omega),x,y)\in F_{\omega_1\cdots\omega_n}$ then
\begin{eqnarray*}
\pi_2(\tilde \psi(\sigma^n(\omega),x,y))&=&\pi_2(\psi(\sigma^n(\omega),x,y))+1\\
&=& R_{\beta,\omega_{n+1}}(x)-1+1=R_{\beta,\omega_{n+1}}(x)
\end{eqnarray*}
as required. The same arguments work for $\overline E$. \end{proof}

\begin{lemma}
The map $\tilde \psi:E\cup\overline E\to E\cup\overline E$ is a bijection which preserves Lebesgue measure $\tilde{\lambda}$.
\end{lemma}
\begin{proof}
We have that $\tilde \psi:=Q\circ \psi$. We have proved that $\psi$ is a measure preserving bijection and so need only to prove that $Q$ is a measure preserving bijection. 

We see that $\psi(F_{\omega_1\cdots\omega_n})=E_{\omega_1\cdots\omega_n 0}\cap\{x\in[0,\frac{1}{\beta-1}-1]\}$. Then we have that
\begin{eqnarray*}
\psi(\overline F_{\overline{\omega_1}\cdots\overline{\omega_n}})&=&E_{\overline{\omega_1}\cdots\overline{\omega_n} 1}\cap\{x\in[1,\frac{1}{\beta-1}]\}\\
&=& E_{\overline{\omega_1}\cdots\overline{\omega_n} \overline 0}\cap\{x-1\in[0,\frac{1}{\beta-1}-1]\}\\
&=& Q(E_{\omega_1\cdots\omega_n 0}\cap\{x\in[0,\frac{1}{\beta-1}-1]\})\\
&=& Q\circ \psi(F_{\omega_1\cdots\omega_n}).
\end{eqnarray*}
Similarly $Q\circ \psi(\overline F_{\overline{\omega_1}\cdots\overline{\omega_n}})=\psi(F_{\omega_1\cdots\omega_n})$. Then we see that $Q$ leaves points unaffected if they are not an element of $\psi(F(\omega_1\cdots\omega_n))$ or $\psi(\overline F_{\overline{\omega_1}\cdots\overline{\omega_n}})$ for some $\omega_1\cdots\omega_n$, whereas it interchanges $\psi(F(\omega_1\cdots\omega_n))$ and $\psi(\overline F_{\overline{\omega_1}\cdots\overline{\omega_n}})$ by a translation and a reflection. Since translation and reflection preserve $\tilde{\lambda}$ we conclude that $Q$ is a measure preserving bijection as required.\end{proof}

Hence we have that the system $(E\cup\overline E,\tilde \psi,\tilde{\lambda})$ is a natural extension of $(\Omega\times I_{\beta},R_{\beta},\hat{\mu}_{\beta})$.

Finally we prove Theorem \ref{densitytheorem}. $E\cup \overline{E}$ is the product of $\Omega$ with a set in $\mathbb R^2$. Then projecting $\tilde{\lambda}=(m_{\frac{1}{2}}\times \lambda\times\lambda)|_{E\cup\overline E}$ onto $\Omega\times I_{\beta}$ we get the measure $m_{\frac{1}{2}}\times \mu_{\beta}*$ where $\mu_{\beta}*$ has density
\[
\int_{\mathbb R} \chi_{E\cup\overline E}(x,y) dy.
\]
Normalising this measure to make it a probability measure gives us the absolutely continuous invariant measure $\hat{\mu}_{\beta}$, and we see that the density of $\mu_{\beta}$ is given by
\begin{eqnarray*}
\rho_{\beta}(x)&=& C(\beta)\int_{\mathbb R} \chi_{E\cup\overline E}(x,y) dy\\
&=& C(\beta) \sum_{n=0}^{\infty}\frac{1}{(2\beta)^n}\left(\sum_{\omega_1\cdots\omega_n\in\{0,1\}^n} \chi_{[0,R_{\beta,\omega_1\cdots\omega_n}^n(1)]}(x)+\chi_{[R_{\beta,\omega_1\cdots\omega_n}^n(\frac{1}{\beta-1}-1),\frac{1}{\beta-1}]}(x)\right).
\end{eqnarray*}
This completes the proof of theorem \ref{densitytheorem}.

\section{Further Questions and Comments}
There are several natural questions arising from the construction of our invariant density. The first relates to the biased measures $\hat{\mu}_{\beta,p}$ which are the product of the $(p,1-p)$ Bernoulli measure on $\Omega$ with an absolutely continuous measure $\mu_{\beta,p}$ on $I_{\beta}$. In this article we dealt only with the unbiased measure $\hat{\mu}_{\beta}=\hat{\mu}_{\beta,\frac{1}{2}}$. It seems that our natural extension cannot easily be adapted to deal with the biased case\footnote{In particular, when we built our second tower and built a natural extension of $K_{\beta}$ using it, some mass was swapped between the two towers using the function $Q$. In the biased case the two towers will be of unequal mass and so $Q$ will not be measure preserving.}, but one might still hope to work out a formula for the invariant density.

{\bf Question 1:} Can one write down a formula for the density of the measures $\mu_{\beta,p}$? Is this continuous as a function of $p$?

A second natural question relates to the entropy of the systems $(\Omega\times I_{\beta},K_{\beta},\hat{\mu}_{\beta})$. Looking at the formula for the density of $\mu_{\beta}$ given in Theorem \ref{densitytheorem}, it seems that there are values of $\beta$ for which $\mu_{\beta_n}$ need not converge to $\mu_{\beta}$ in the weak$^*$ topology for sequences $\beta_n\to\beta$. In particular, there should be such a discontinuity whenever $\beta$ is such that \[K^n_{\beta}\left(\omega,\frac{1}{\beta}\right)=\left(\omega',\frac{1}{\beta(\beta-1)}\right)\] for some value of $n\in\mathbb N$ and $\omega,\omega'\in\Omega$. This should cause a corresponding discontinuity in the metric entropy $H$.

{\bf Question 2:} Can one characterise the values of $\beta$ for which the functions $\beta\to\mu_{\beta}$ and $\beta\to H_{\hat{\mu}_{\beta}}(K_{\beta})$ are discontinuous? 

Finally we have two questions about counting beta expansions. We recall that in \cite{CountingBeta} we studied the number of words of length $n$ which can be extended to $\beta$-expansions of $x$ for typical $x$. We defined
\[
\mathcal E^n_{\beta}(x):=\{(x_1,\cdots,x_n) \in \{0,1\}^n|\exists (x_{n+1},x_{n+2},\cdots) : x=\sum_{k=1}^{\infty}x_k\beta^{-k}\}
\]
and studied the quantity $\mathcal N_n(x;\beta):=|\mathcal E^n_{\beta}(x)|$. We demonstrated that, if one understands how $\mathcal N_n(x;\beta)$ grows for typical $x$ as $n\to\infty$, then one can say whether the corresponding Bernoulli convolution is absolutely continuous. In particular, if the function
\[
\liminf_{n\to\infty}\left(\frac{\beta}{2}\right)^n\mathcal N_n(x;\beta)
\]
has positive integral then $\nu_{\beta}$ is absolutely continuous. We were able to give an explicit formula for $\mathcal N_n(x;\beta)$ in terms of $K_{\beta}$:
\begin{equation}\label{eq11}
\mathcal N_n(x;\beta)=\int_{\{0,1\}^{\mathbb N}}2^{h(\omega,x,n)}dm
\end{equation}
where $m$ is the $(\frac{1}{2},\frac{1}{2})$ Bernoulli measure on $\Omega$ and
\[
h(\omega,x,n):=\#\{i \in \{1,\cdots, n\}: K_{\beta}^i(\omega,x) \in \Omega\times[\frac{1}{\beta},\frac{1}{\beta(\beta-1)}]\}.
\]
However we were only able to use the above formula to get a lower bound for the growth rate of $\mathcal N_n(x;\beta)$, we were able to show that
\begin{equation}\label{eq22}
\liminf_{n\rightarrow\infty} \dfrac{\log (\mathcal N_n(x;\beta))}{n}\geq \log(2)\mu_{\beta}(S).
\end{equation}
Using our formula for the density of $\mu_{\beta}$ we can get explicit bounds on $\mu_{\beta}(S)$, and hence lower bounds on the growth rate of $\mathcal N_n(x;\beta)$, but these are not strong enough to ascertain whether a given Bernoulli convolution is absolutely continuous or not. There are however some natural questions which we can ask.

{\bf Question 3:} The ergodic theory taking one from equation \ref{eq11} to the inequality \ref{eq22} is rather crude, can one combine the work in this article on $\mu_{\beta}$ with central limit theorems and information on higher moments for $K_{\beta}$ to improve inequality \ref{eq22}?

{\bf Question 4:} Do the values of $\beta$ at which the function $\beta\to\mu_{\beta}$ is not weak$^*$ continuous have any significance in the study of Bernoulli convolutions?

\section*{Acknowledgements}
Many thanks to Karma Dajani for the many interesting discussions relating the construction of our natural extension. This work was supported by the Dutch Organisation for Scientific Research (NWO) grant number 613.001.022.

\bibliographystyle{plain} 
\bibliography{betaref.bib}

\begin{thebibliography}{10}

\bibitem{BakerGolden}
S.~{Baker}.
\newblock {Generalised golden ratios over integer alphabets}.
\newblock {\em ArXiv e-prints}, October 2012.

\bibitem{BakerGrowth}
S.~{Baker}.
\newblock {The growth rate and dimension theory of beta-expansions}.
\newblock {\em ArXiv e-prints}, August 2012.

\bibitem{DajaniKalleDeleted}
K.~Dajani and C.~Kalle.
\newblock A natural extension for the greedy {$\beta$}-transformation with
  three arbitrary digits.
\newblock {\em Acta Math. Hungar.}, 125(1-2):21--45, 2009.

\bibitem{DKS}
K.~Dajani, C.~Kraaikamp, and B.~Solomyak.
\newblock The natural extension of the {$\beta$}-transformation.
\newblock {\em Acta Math. Hungar.}, 73(1-2):97--109, 1996.

\bibitem{DdV1}
Karma Dajani and Martijn de~Vries.
\newblock Measures of maximal entropy for random {$\beta$}-expansions.
\newblock {\em J. Eur. Math. Soc. (JEMS)}, 7(1):51--68, 2005.

\bibitem{DdV}
Karma Dajani and Martijn de~Vries.
\newblock Invariant densities for random {$\beta$}-expansions.
\newblock {\em J. Eur. Math. Soc. (JEMS)}, 9(1):157--176, 2007.

\bibitem{DKRandom}
Karma Dajani and Cor Kraaikamp.
\newblock Random {$\beta$}-expansions.
\newblock {\em Ergodic Theory Dynam. Systems}, 23(2):461--479, 2003.

\bibitem{FengSidorov}
De-Jun Feng and Nikita Sidorov.
\newblock Growth rate for beta-expansions.
\newblock {\em Monatsh. Math.}, 162(1):41--60, 2011.

\bibitem{JordanShmerkinSolomyak}
Thomas Jordan, Pablo Shmerkin, and Boris Solomyak.
\newblock Multifractal structure of {B}ernoulli convolutions.
\newblock {\em Math. Proc. Cambridge Philos. Soc.}, 151(3):521--539, 2011.

\bibitem{CountingBeta}
Tom Kempton.
\newblock Counting beta expansions and the absolute continuity of {B}ernoulli
  convolutions.
\newblock {\em ArXiv e-prints}, March 2012.

\bibitem{Parry}
W.~Parry.
\newblock On the {$\beta $}-expansions of real numbers.
\newblock {\em Acta Math. Acad. Sci. Hungar.}, 11:401--416, 1960.

\bibitem{SolomyakSixty}
Yuval Peres, Wilhelm Schlag, and Boris Solomyak.
\newblock Sixty years of {B}ernoulli convolutions.
\newblock In {\em Fractal geometry and stochastics, {II}
  ({G}reifswald/{K}oserow, 1998)}, volume~46 of {\em Progr. Probab.}, pages
  39--65. Birkh\"auser, Basel, 2000.

\bibitem{Renyi}
A.~R{\'e}nyi.
\newblock Representations for real numbers and their ergodic properties.
\newblock {\em Acta Math. Acad. Sci. Hungar}, 8:477--493, 1957.

\bibitem{Sidorov1}
Nikita Sidorov.
\newblock Almost every number has a continuum of {$\beta$}-expansions.
\newblock {\em Amer. Math. Monthly}, 110(9):838--842, 2003.

\bibitem{Sidorovlmt}
Nikita Sidorov.
\newblock Expansions in non-integer bases: lower, middle and top orders.
\newblock {\em J. Number Theory}, 129(4):741--754, 2009.

\end{thebibliography}

\end{document}